\documentclass[11pt, a4paper, oneside]{amsart}

\usepackage[english]{babel}
\usepackage{amsmath, amsthm, amsfonts, mathrsfs, amssymb}
\usepackage{mathtools}
\mathtoolsset{centercolon}
\usepackage{booktabs}
\usepackage[shortlabels]{enumitem}
\setlist[itemize]{leftmargin=20pt}
\setlist[enumerate]{leftmargin=20pt}
\usepackage[colorlinks, citecolor = blue]{hyperref}
\usepackage[hmargin=3cm,vmargin=3cm]{geometry}

\usepackage{color}
\usepackage{graphicx}
\usepackage{tikz-cd}
\usetikzlibrary{arrows}

\newcommand{\Q}{\ensuremath{\mathbf{Q}}}
\newcommand{\R}{\ensuremath{\mathbf{R}}}
\newcommand{\C}{\ensuremath{\mathbf{C}}}

\newcommand{\mc}{\mathcal}

\DeclarePairedDelimiter\abs{\lvert}{\rvert}

\DeclarePairedDelimiter\cbrace\{\}
\DeclarePairedDelimiter\ha()

\DeclarePairedDelimiter{\nrm}\lVert\rVert

\DeclarePairedDelimiter{\nnrm}{\lvert\kern-0.3ex\lvert\kern-0.3ex\lvert}{\rvert\kern-0.3ex\rvert\kern-0.3ex\rvert}

\newcommand{\nrmb}[1]{\bigl\|#1\bigr\|}
\newcommand{\absb}[1]{\bigl|#1\bigr|}
\newcommand{\hab}[1]{\bigl(#1\bigr)}
\newcommand{\cbraceb}[1]{\bigl\{#1\bigr\}}

\newcommand{\nnrmb}[1]{\bigl\vert\kern-0.3ex\bigl\vert\kern-0.3ex\bigl\vert#1\bigr\vert\kern-0.3ex\bigr\vert\kern-0.3ex\bigr\vert}

\newcommand{\has}[1]{\Bigl(#1\Bigr)}
\newcommand{\cbraces}[1]{\Bigl\{#1\Bigr\}}

\newcommand{\nnrms}[1]{\Bigl\vert\kern-0.3ex\Bigl\vert\kern-0.3ex\Bigl\vert#1\Bigr\vert\kern-0.3ex\Bigr\vert\kern-0.3ex\Bigr\vert}

\DeclareMathOperator{\ind}{\mathbf{1}}

\newcommand{\vvvert}[1]{{\left\vert\kern-0.3ex\left\vert\kern-0.3ex\left\vert #1 
    \right\vert\kern-0.3ex\right\vert\kern-0.3ex\right\vert}}
\newcommand{\dd}{\hspace{2pt}\mathrm{d}}
\newcommand{\ddn}{\mathrm{d}}
\newcommand{\ee}{\mathrm{e}}

\def\avint_#1{\mathchoice{\mathop{\kern 0.2em\vrule width 0.6em height 0.69678ex depth -0.58065ex \kern -0.8em \intop}\nolimits_{\kern -0.4em#1}}{\mathop{\kern 0.1em\vrule width 0.5em height 0.69678ex depth -0.60387ex \kern -0.6em \intop}\nolimits_{#1}} {\mathop{\kern 0.1em\vrule width 0.5em height 0.69678ex depth -0.60387ex \kern -0.6em \intop}\nolimits_{#1}} {\mathop{\kern 0.1em\vrule width 0.5em height 0.69678ex depth -0.60387ex \kern -0.6em \intop}\nolimits_{#1}}}

\newtheorem{theorem}{Theorem}
\newtheorem{corollary}[theorem]{Corollary}
\newtheorem{lemma}[theorem]{Lemma}
\newtheorem{proposition}[theorem]{Proposition}

{}

\theoremstyle{remark}
\newtheorem{remark}[theorem]{Remark}
\newtheorem{example}[theorem]{Example}

\theoremstyle{definition}
\newtheorem{definition}[theorem]{Definition}

\numberwithin{theorem}{section}
\numberwithin{equation}{section}
\title[Banach function spaces done right]{Banach function spaces done right}
\thanks{Z. N. is supported by the grant Juan de la Cierva formación 2021 FJC2021-046837-I, the Basque Government through the BERC 2022-2025 program, by the Spanish State Research Agency project PID2020-113156GB-I00/AEI/10.13039/501100011033 and through BCAM Severo Ochoa excellence accreditation SEV-2023-2026.}

\author{Emiel Lorist}
\address[E. Lorist]{Delft Institute of Applied Mathematics \\ Delft University of Technology \\ P.O. Box 5031\\ 2600 GA Delft \\The Netherlands}
\email{e.lorist@tudelft.nl}

\author{Zoe Nieraeth}
\address[Z. Nieraeth]{BCAM\textendash  Basque Center for Applied Mathematics, Bilbao, Spain}
\email{zoe.nieraeth@gmail.com}
\allowdisplaybreaks

\begin{document}
\begin{abstract}
In this survey, we discuss the definition of a (quasi-)Banach function space. We advertise the original definition by Zaanen and Luxemburg, which does not have various issues introduced by other, subsequent definitions. Moreover, we prove versions of well-known basic properties of Banach function spaces in the setting of \emph{quasi}-Banach function spaces.
\end{abstract}




\maketitle

\section{Introduction}
Banach spaces of measurable functions, like (weighted) Lebesgue spaces, Orlicz spaces, Lorentz spaces, Morrey spaces and tent spaces play a central role in many areas of mathematical analysis. These spaces all fall within the broader class of \emph{Banach function spaces}, which have the property that the pointwise order between functions is in some sense compatible with the norm.

Various definitions of Banach function spaces exist in the literature. A popular choice  can be phrased as follows: A Banach function space $X$ is a subspace of $L^0(\Omega)$,
the space of measurable functions $f \colon \Omega \to \C$ for a $\sigma$-finite measure space $(\Omega,\mu)$, equipped with a norm $\nrm{\,\cdot\,}_X$ such that it satisfies the following properties:
\begin{itemize}
  \item \textit{Ideal property:} If $f\in X$ and $g\in L^0(\Omega)$ with $|g|\leq|f|$ a.e., then $g\in X$ with $\nrm{g}_X\leq \nrm{f}_X$;
\item \textit{Fatou property:} If $0\leq f_n \uparrow f$ for $(f_n)_{n\geq 1}$ in $X$ and $\sup_{n\geq 1}\nrm{f_n}_X<\infty$, then $f \in X$ and $\nrm{f}_X=\sup_{n\geq 1}\nrm{f_n}_X$;
\end{itemize}
the latter of which implies that $X$ is complete. Moreover, to ensure that $X$ contains a sufficient number of functions, it is assumed that, for any measurable set $E \subseteq \Omega$ of finite measure, one has
\begin{align}\label{eq:illegal}
\ind_{E} &\in X,\\
\label{eq:illegaldual}
\int_E \abs{f}\dd \mu<\infty &\text{ for all }  f\in X,
\end{align}
the latter being equivalent to the existence of a $C_E>0$ such that $\int_E \abs{f} \dd \mu \leq C_E \nrm{f}_X$.
The definition of a Banach function space with these properties is often attributed to the book by Bennet and Sharpley \cite{BS88}. Some variants already appeared earlier in the book by Lindenstrauss and Tzafriri \cite{LT79}.

The ideal property is the most fundamental property of a Banach function space, making sure that the natural partial order on $L^0(\Omega)$ is compatible with the norm on $X$. 
The Fatou property can be omitted in the definition of a Banach function space. In this case, one has to ensure the completeness of $X$ separately, either by assuming it explicitly or through a notion called the \emph{Riesz--Fischer property}. This is  the approach which we take in this survey, see Subsection \ref{subs:completeness}. 

Originally, the Fatou property was introduced as part of the definition in the PhD thesis of Luxemburg \cite{Lu55}, but it was  later removed in the series of papers by Luxemburg and Zaanen \cite{LZ63-1} and the subsequent book by Zaanen \cite{Za67}. Unfortunately, it was then reintroduced  in \cite{BS88}. To give an example where this is problematic, proper closed subspaces of Banach function spaces such as, e.g., $c_0\subsetneq\ell^\infty$ do not satisfy the Fatou property (see Proposition~\ref{prop:noproperclosedsubspace}). Nonetheless, for example in applications in harmonic analysis, this situation is somewhat pathological. Indeed, any space of functions with the ideal property, but without the Fatou property, can be continuously embedded in a space that \emph{does} have the Fatou property (see Proposition~\ref{prop:fatouembedding}). Moreover, the Fatou property ensures that the integral pairing with functions in the K\"othe dual $X'$,
i.e. the space
$$
X':=\{g\in L^0(\Omega):fg\in L^1(\Omega)\text{ for all $f\in X$}\}.
$$
with
\[
\|g\|_{X'}:=\sup_{\substack{\|f\|_X=1}}\|fg\|_{L^1(\Omega)},
\]
recovers the norm of $X$. Since this allows for the use of  duality arguments typical in many areas of mathematical analysis, the Fatou property is therefore highly desirable.

The main problem with the above definition of a Banach function space from \cite{BS88}, when working in areas such as harmonic analysis, are properties
\eqref{eq:illegal} and \eqref{eq:illegaldual}. For example, the weighted Lebesgue space $L^p(\R^d,w)$ for a Muckenhoupt weight $w \in A_p$ may not be a Banach function space over $\R^d$ with the Lebesgue measure in the sense of the definition stated above, see \cite[Section 7.1]{SHYY17}. To circumvent this issue, in the work \cite{CMM22} the authors included these spaces by considering \eqref{eq:illegal} and \eqref{eq:illegaldual} with respect to the measure $w\,\mathrm{d}x$ rather than with respect to the Lebesgue measure. This, however, is inadequate, as it still does not include many important spaces, such as Musielak--Orlicz spaces and  Morrey spaces, even in the unweighted setting.
\begin{itemize}
\item  A Musielak--Orlicz space $L^\varphi(\Omega)$, with $\varphi\colon \Omega \times \R_+\to \R_+$ such that $\int_E\varphi(s,t) \dd\mu(s)=\infty$ for some measurable $E \subseteq \Omega$ of finite measure and all $t>0$, does not satisfy \eqref{eq:illegal} (see, e.g., \cite[Chapter II]{Mu83}).
\item Certain Morrey spaces do not satisfy \eqref{eq:illegaldual}, as was shown in \cite[Example~3.3]{ST15b}.
\end{itemize}
In recent literature, this issue has lead the authors of \cite{MMMMM22} to develop certain theory for Morrey spaces in Chapter 7 and afterwards prove analogous results in Banach function spaces (with assumptions \eqref{eq:illegal} and \eqref{eq:illegaldual}) in Chapter 8. The results in Chapter 7 could have been regarded as a special case of the results in Chapter 8 if a definition of Banach function spaces that includes Morrey spaces would have been chosen.

A second issue with \eqref{eq:illegal} and \eqref{eq:illegaldual} arises when one wants to treat \emph{quasi}-Banach function spaces, i.e. replacing the norm on $X$ by a quasi-norm. In this setting, the condition \eqref{eq:illegaldual} is typically far too restrictive, as can already be seen when considering $L^p(\R^d)$ for $0<p<1$. However, omitting only \eqref{eq:illegaldual} leads to the asymmetric situation in which the K\"othe dual $X'$
does not necessarily contain the required indicator functions (see \cite[Proposition 2.16]{ORS08} for an illustration of this phenomenon). Moreover, omitting both \eqref{eq:illegal} and \eqref{eq:illegaldual} instead  also leads to pathological situations, as $\nrm{\,\cdot\,}_{X'}$ may only be a semi-norm in this case, see Subsection~\ref{subsec:saturation}.

\bigskip

Recognizing these problems with the definition of a (quasi-)Banach function space including \eqref{eq:illegal} and \eqref{eq:illegaldual}, the authors of \cite{SHYY17} proposed a solution to these issues by introducing so-called \emph{ball} quasi-Banach function spaces, in which the arbitrary measurable sets $E$ in \eqref{eq:illegal} and \eqref{eq:illegaldual} are replaced by metric balls. This definition has since been adopted by various authors, see, e.g., \cite{DGPYYZ22, DLYYZ23, PYYZ23b, SYY22, TYYZ21, WYY23, We22}. However, morally speaking, the definition of a (quasi-)Banach function space should be a measure theoretic one, i.e. not referencing any metric structure of $\Omega$. This is, for example, of paramount importance when working on the intersection between harmonic analysis and probability theory, as the natural object to work with in that setting is a probability space without any metric structure. 

Furthermore, there is no need to define a new notion (like a ball quasi-Banach function space) in order to solve the issue with the assumptions in \eqref{eq:illegal} and \eqref{eq:illegaldual}. Indeed, the solution is readily available in the literature, dating all the way back to the works of Zaanen and Luxemburg \cite{Lu55, LZ63-1,Za67}. Indeed, one should replace \eqref{eq:illegal} and \eqref{eq:illegaldual} by the assumption that $X$ is \emph{saturated}:
\begin{itemize}
  \item \textit{Saturation property:} For every measurable $E\subseteq\Omega$ of positive measure, there exists a measurable $F\subseteq E$ of positive measure with $\ind_F\in X$.
\end{itemize}
Defining quasi-Banach function spaces using the ideal and saturation properties yields a purely measure-theoretic definition, which includes all aforementioned specific function spaces as examples:
\begin{itemize}
\item For weighted Lebesgue spaces and Morrey spaces, one can use $F =E \cap B$ for a large enough ball in $\R^d$.
\item For Musielak--Orlicz spaces, one can use $F=E \cap T_n$ for large enough $n$ and $T_n$ as in \cite[p.64]{Kam85}. 
\end{itemize}
Moreover, as we shall see, the K\"othe dual of such a space automatically satisfies the ideal, Fatou, and, if $X$ is a Banach function space, the saturation properties.

It should be noted that the saturation property has various equivalent formulations. It is, for example, equivalent to either of the following assumptions (see Proposition \ref{prop:weakorderunit}):
\begin{enumerate}[(i)]
    \item\label{it:sat1intro}  There exists a $u \in X$ with $u>0$ a.e.;
    \item\label{it:sat2intro}   There is an increasing sequence of sets $F_n\subseteq\Omega$ with $\ind_{F_n}\in X$ and $\bigcup_{n=1}^\infty F_n=\Omega$;
\end{enumerate}
The function $u$ in assumption \ref{it:sat1intro} is called a \emph{weak order unit}. Generally, its utility is in that the ideal property of $X$ implies that $u\ind_E\in X$ for \emph{all} measurable sets $E\subseteq\Omega$. Thus, arguments that require \eqref{eq:illegal} can still be done by simply multiplying each function in the space by $u^{-1}$. We detail this procedure in Section~\ref{subsec:weights}. As a matter of fact, there is a weight $0<w\in L^1(\Omega)$ so that, with respect to the measure $w\,\mathrm{d}\mu$, 
 the condition \eqref{eq:illegaldual} is also satisfied by this space (see Proposition~\ref{prop:weightkothe}).

 The assumption in \ref{it:sat2intro} is actually the assumption used in \cite{Lu55}. Notably, almost $70$ years later, the authors of the recent book \cite{MMM23b} seem to have independently rediscovered the exact formulation of the assumption \ref{it:sat2intro}, calling the resulting class of spaces \emph{generalized Banach function spaces}. However, it would historically be more accurate to refer to this class simply as Banach function spaces, whereas the class of spaces with properties \eqref{eq:illegal} and \eqref{eq:illegaldual} should be called \emph{restricted Banach function spaces}.

\bigskip

The goal of this survey is two-fold. 
\begin{itemize}
  \item First of all, we would like to advertise the definition of a (quasi-)Banach function space using the saturation property instead of \eqref{eq:illegal} and \eqref{eq:illegaldual} and the Fatou  property as optional assumption.
  \item Secondly, we will provide versions of well-known basic properties of Banach function spaces in the setting of \emph{quasi}-Banach function spaces.
\end{itemize}
 Our claim to originality in this survey is rather humble. Most of our discussion for Banach function spaces can, for example, also be found in \cite[Chapter 15]{Za67}. However, we are not aware of a comprehensive reference work for the \emph{quasi}-Banach function space case (see, e.g., \cite{KM07, KMP03, Vo15, NP20} for some partial results), and hope that this survey  may serve as a solid introduction for anyone working with quasi-Banach function spaces. In particular,  multilinear harmonic analysis has in recent years become a very active research area, in which the quasi-Banach range naturally makes its appearance.

\section{Quasi-Banach function spaces}\label{sec:qBFS}
In this section we will introduce quasi-Banach function spaces and discuss their defining properties in detail.
Let $(\Omega,\mu)$ be a  measure space, which will always be assumed to be $\sigma$-finite. Let $L^0(\Omega)$ denote the space of measurable functions on $(\Omega,\mu)$. Let $X \subseteq L^0(\Omega)$ be a complete, quasi-normed vector space. Denote the quasi-norm by $\nrm{\,\cdot\,}_X$ and the optimal constant $K\geq 1$ such that
\[
\nrm{f+g}_X\leq K(\nrm{f}_X+\nrm{g}_X), \qquad f,g \in X,
\]
 by $K_X$. The space $X$ called a \emph{quasi-Banach function space over $(\Omega,\mu)$} if it satisfies the following properties:
\begin{itemize}
  \item \textit{Ideal property:} If $f\in X$ and $g\in L^0(\Omega)$ with $|g|\leq|f|$, then $g\in X$ with $\nrm{g}_X\leq \nrm{f}_X$.
  \item \textit{Saturation property:} For every measurable $E\subseteq\Omega$ of positive measure, there exists a measurable $F\subseteq E$ of positive measure with $\ind_F\in X$.
\end{itemize}
If $\nrm{\,\cdot\,}_X$ is a norm, i.e., if $K_X=1$, then $X$ is called a \emph{Banach function space over $(\Omega,\mu)$.} 
Since the ideal property is inherently tied to the choice of quasi-norm $\|\cdot\|_X$ on the space $X$, we sometimes emphasize this by writing $(X,\|\cdot\|_X)$ rather than $X$.

\begin{remark}\label{rem:def}~
\begin{enumerate}[(i)]
\item Instead of introducing a quasi-Banach function space as a complete quasi-normed space with the ideal and saturation properties, one can equivalently start by defining a \emph{function quasi-norm} $\rho\colon L^0(\Omega)_+ \to [0,\infty]$ satisfying corresponding versions of these properties and afterwards setting
$$
X:=\cbrace{f \in L^0(\Omega): \rho\ha{\abs{f}}<\infty}, \qquad \nrm{f}_X:= \rho(\abs{f}).
$$
The equivalence of these approaches can be seen 
by setting $$\rho(f) := \begin{cases}
    \nrm{f}_X, \qquad &f \in X,\\
    \infty, &f \notin X.
\end{cases}$$
\item\label{it:def2} As we will show in Section~\ref{subs:completeness}, a quasi-normed vector space $X\subseteq L^0(\Omega)$  is complete if and only if it has the Riesz--Fischer property:
\begin{itemize}
     \item \textit{Riesz--Fischer property:} If $(f_n)_{n\geq 1}$ in $X$ and $\sum_{n=1}^\infty K_X^n\|f_n\|_X<\infty$, then $\sum_{n=1}^\infty f_n\in X$ with $\big\|\sum_{n=1}^\infty f_n\big\|_X\leq K_X\sum_{n=1}^\infty K_X^n\|f_n\|_X$.
\end{itemize}
In many examples,  $X$ actually satisfies the stronger Fatou property:
\begin{itemize}
\item \textit{Fatou property:} If $0\leq f_n \uparrow f$ for $(f_n)_{n\geq 1}$ in $X$ and $\sup_{n\geq 1}\nrm{f_n}_X<\infty$, then $f \in X$ and $\nrm{f}_X=\sup_{n\geq 1}\nrm{f_n}_X$.
\end{itemize}
One readily checks that the Fatou property implies the Riesz--Fischer property and thus completeness. Indeed, by the quasi-triangle inequality and induction on $N\geq 1$ we have
\[
\Big\|\sum_{n=1}^N f_n\Big\|_X\leq \sum_{n=1}^N K_X^n\|f_n\|_X.
\]
The Riesz--Fischer property then follows by using the Fatou property on the partial sums $\sum_{n=1}^N |f_n|$. 
\item     In some parts of the literature, the underlying measure space $(\Omega,\mu)$ is assumed to be complete. We do not assume completeness as this assumption is superfluous in the following sense. Suppose that $(\Omega,\mu)$, with the $\sigma$-algebra $\Sigma$, is \emph{not} complete. Denoting its completion (which is again $\sigma$-finite) by $\Sigma^\ast$ with measure $\mu^\ast$, there is a natural one-to-one correspondence between the measurable functions with respect to $\Sigma$ and with respect to $\Sigma^\ast$. Indeed, for each $f^\ast$ that is measurable with respect to $\Sigma^\ast$ there exists an $f$ that is measurable with respect to $\Sigma$ such that $f^\ast=f$ $\mu$-a.e. Thus, any quasi-Banach function space over $(\Omega,\mu)$ may as well be considered over $(\Omega,\mu^\ast)$.
    \end{enumerate}
\end{remark}

By the Aoki--Rolewicz theorem \cite{Ao42,Ro57} 
\begin{equation} \label{eq:Aoki}
\nnrm{f} := \inf\cbraces{\has{\sum_{k=1}^n \nrm{f_k}_X^p}^{1/p}: f_1,\ldots,f_n \in X \text{ such that } \sum_{k=1}^n f_k = f}
\end{equation}
is an equivalent \emph{$p$-norm} on $X$ for $p \in (0,1]$ with $2^{1/p} = 2K_X$, i.e. $\nnrm{\,\cdot\,}$ is a quasi-norm on $X$ such that
\begin{align*}
\vvvert{f+g}^p &\leq \vvvert{f}^p + \vvvert{ g}^p, &f,g\in X\\
4^{-1/p}\nrm{f}_X &\leq \nnrm{f} \leq \nrm{f}_X, &f \in X,
\end{align*}
see, e.g., \cite{KPR84}. It is a straightforward check to see that $(X,\nnrm{\,\cdot\,})$ is again a quasi-Banach function space.

\subsection{Completeness}\label{subs:completeness}
Let us discuss the defining properties of a quasi-Banach function space $X$ in some detail. To start, we note that the assumed completeness can be reformulated as the Riesz-Fischer property (see Remark \ref{rem:def}\ref{it:def2} for the definition). Indeed, if a quasi-Banach space $X$ is complete, then it satisfies the Riesz--Fischer property. To see this, note that by the quasi-triangle inequality and induction on $N\geq 1$ we have
\[
\Big\|\sum_{n=1}^N f_n\Big\|_X\leq \sum_{n=1}^N K_X^n\|f_n\|_X.
\]
A standard argument then shows that the partial sums $F_N:=\sum_{n=1}^N f_n$ are a Cauchy sequence in $X$, proving that $F:=\sum_{n=1}^\infty f_n\in X$. The assertion then follows from noting that 
\[
\|F\|_X\leq K_X\|F-F_N\|_X+K_X\sum_{n=1}^\infty K_X^n\|f_n\|_X
\]
and letting $N\to\infty$. The converse statement is also true. For a proof, we refer the reader to \cite[Theorem 1.1]{Ma04}.
\begin{proposition}\label{prop:complete}
Let $X$ be a quasi-normed space. Then $X$ is complete if and only if $X$ satisfies the Riesz--Fischer property.
\end{proposition}

For $X \subseteq L^0(\Omega)$ the Riesz--Fischer property ensures that convergence in the norm of $X$ implies local convergence in measure, i.e., the embedding $X\hookrightarrow L^0(\Omega)$ is continuous. As $L^0(\Omega)$ equipped with the topology of local convergence in measure is a Hausdorff space, this ensures uniqueness of limits. More precisely, convergence in the quasi-norm of $X$ implies
 pointwise a.e. convergence for a subsequence, so that the pointwise a.e. limit and the limit in the quasi-norm of $X$ coincide whenever both exist.
 \begin{proposition}\label{prop:convinmeasure}
    Let $X$ be a quasi-Banach function space over $(\Omega,\mu)$, let $(f_n)_{n\geq 1}$ be a sequence in $X$ and let $f \in X$.
    \begin{enumerate}[(i)]
        \item\label{it:conv1} If $(f_n)_{n\geq 1}$ is Cauchy in $X$, then $(f_n)_{n\geq 1}$ is locally Cauchy in measure.
        \item\label{it:conv2}  If $(f_n)_{n\geq 1}$ converges to $f$ in $X$, then $(f_n)_{n\geq 1}$ converges locally in measure to $f$.
    \end{enumerate}
In particular, if $f_n \to f$ in $X$ and $f_n \to g $ pointwise a.e., then $f=g$ a.e.
\end{proposition}

\begin{proof}
We will only prove \ref{it:conv1}, the proof of \ref{it:conv2}  is similar. Fix a measurable $E \subseteq \Omega$ with $\mu(E)<\infty$, let $\varepsilon>0$ and define for $j,k \geq 1$
$$
A_{j,k} = \cbrace{x \in E: \abs{f_j(x)-f_k(x)}\geq \varepsilon}.
$$
We need to show that $\lim_{j,k \to \infty} \mu\ha{A_{j,k}}=0$.

Since $(f_n)_{n\geq 1}$ is Cauchy in $X$, we know by the ideal property that
$$
\lim_{j,k \to \infty} \nrm{\ind_{A_{j,k}}}_X \leq \lim_{j,k \to \infty} \varepsilon^{-1} \nrm{f_j-f_k}_X = 0.
$$
Suppose that $\mu(A_{j,k})\not\to 0$ for $j,k \to \infty$. Then we can find a $\delta>0$ and a sequence of measurable sets $(B_n)_{n\geq 1}$ in $\{A_{j,k}\}$ such that $\mu(B_n) \geq \delta$ for all $n\geq 1$ and $\nrm{\ind_{B_n}}_X \to 0$ for $n\to \infty$. By considering a subsequence if necessary, we may furthermore assume that $\nrm{\ind_{B_n}}_X  \leq 2^{-n}K_X^{-n-1}$ for all $n\geq 1$.
By the Riesz--Fischer property, it follows for $m\geq 1$ that $\ind_{\bigcup_{n=m}^\infty B_{n}} \in X$ with
$$
\nrmb{\ind_{\bigcup_{n=m}^\infty B_{n}}}_X \leq K_X\sum_{n=m}^\infty K_X^{n-m+1} \nrm{\ind_{B_{n}}}_X \leq \sum_{n=m}^\infty 2^{-n} = 2^{-m+1}.
$$

Define $B = \bigcap_{m=1}^\infty \bigcup_{n=m}^\infty B_n$. 
Then we have, by the ideal property, that for all $m\geq 1$
$$
\nrm{\ind_{B}}_X \leq  \nrmb{\ind_{\bigcup_{n=m}^\infty B_{n}}}_X\leq  2^{-m+1}
$$
and thus $\nrm{\ind_B}_X = 0$. This means that $\ind_B =0$ a.e. and consequently $\mu\ha{B}=0$. But $\mu\ha{B_n} \geq \delta$ for all $n\geq 1$ and therefore $\mu\hab{\bigcup_{n=m}^\infty B_{n}} \geq \delta$ for all $m\geq 1$. Since $$\mu\has{\bigcup_{n=1}^\infty B_{n}} \leq \mu\ha{E}<\infty,$$ we conclude that  $\mu\ha{B}\geq \delta$, a contradiction. Thus, we must have $\lim_{j,k \to \infty} \mu\ha{A_{j,k}}=0$.
\end{proof}

 \begin{remark} The \emph{local} Cauchy (or convergence) in measure in Proposition \ref{prop:convinmeasure} can not be replaced by \emph{global} Cauchy (or convergence) in measure. Indeed, take $\Omega = \R$ equipped with the Lebesgue measure $\ddn x$ and define $w(x) := \ee^{-\abs{x}}$. Let $X = L^1_w(\R)$, i.e. the space of all $f \in L^0(\R)$ such that
     $$
     \nrm{f}_{L^1_w(\R)} := \int_\R \abs{f}w\dd x <\infty.
     $$
     The sequence of functions $(\ind_{[n,n+1]})_{n\geq 1}$ converges to zero in $X$, but
     $$
    \absb{\{x \in \R: \abs{\ind_{[j,j+1]}(x)-\ind_{[k,k+1]}(x)}\geq 1\}} =2 
     $$
     for $j\neq k$, i.e. $\ind_{[n,n+1] }$ is not globally Cauchy in measure.  
 \end{remark}

\subsection{The ideal property}\label{subs:idealprop}
The first property of a quasi-Banach function space is called the ideal property, since it ensures that $X$ is a so-called \emph{order ideal} in the vector lattice $L^0(\Omega)$. This implies that $X$
is a quasi-Banach lattice, which explains why a quasi-Banach function space is sometimes called an \emph{ideal quasi-Banach lattice of functions} \cite{KPS82}. In particular, any result for (quasi-)Banach lattices also holds for (quasi-)Banach function spaces.
We refer the reader to e.g. \cite{LT79,MN91,Sc74} for a thorough study of Banach lattices. 

Furthermore, let us note that it follows from the ideal property that for any $f \in L^0(\Omega)$ we have $f\in X$ if and only if $|f|\in X$ with $\nrm{|f|}_X=\nrm{f}_X$.

 \subsection{The saturation property}\label{subsec:saturation}
The second property of a quasi-Banach function space $X$, the saturation property, has already been discussed in the introduction. It is imposed to avoid trivialities. Indeed, 
it ensures that there are no measurable sets $E \subseteq \Omega$ on which all elements of $X$ vanish a.e. This assumption is not restrictive, as any quasi-normed vector space of measurable functions on $\Omega$ can be made saturated by removing the part of $\Omega$ on which all functions in $X$ vanish a.e. (cf. \cite[Section 67]{Za67}). 

To illustrate this, consider the space $X\subseteq L^0(\R)$ defined as the subspace of $L^1(\R)$ consisting of the integrable functions supported in $[0,1]$. When equipped with the norm 
\[
\|f\|_X:=\int_0^1\!|f|\,\mathrm{d}x,
\]
the space $(X, \nrm{\cdot}_X)$ is complete and satisfies the ideal property, but not the saturation property. We also note that any $g\in L^0(\R)$ supported outside of $[0,1]$ satisfies $\int_\R\!|fg|\,\mathrm{d}x=0$ for all $f\in X$. These problems can easily be rectified by considering this as a space over $\Omega=[0,1]$ rather than over $\Omega=\R$, in which case it is saturated, and we simply have $X=L^1([0,1])$.

\bigskip

There are various equivalent formulations of the saturation property. Especially the existence of a weak order unit is often useful in applications, see also Subsection \ref{subsec:weights}.
\begin{proposition}\label{prop:weakorderunit}
Let $(\Omega,\mu)$ be a $\sigma$-finite measure space and let $X \subseteq L^0(\Omega)$ be a quasi-Banach space satisfying the ideal property.
Then the following are equivalent:
\begin{enumerate}[(i)]
\item \label{it:propsat1} $X$ satisfies the saturation property;
\item\label{it:propsat2} There is an increasing sequence of sets $F_n\subseteq\Omega$ with $\ind_{F_n}\in X$ and $\bigcup_{n=1}^\infty F_n=\Omega$;
\item\label{it:propsat3} $X$ has a \emph{weak order unit}, i.e., there is a $u\in X$ with $u>0$ a.e.;
\item\label{it:propsat4} If $g \in L^0(\Omega)$ with $\int_\Omega\!|fg|\,\mathrm{d}\mu= 0$ for all $f \in X$, then $g = 0$ a.e.
\end{enumerate}
\end{proposition}
\begin{proof}
We start by proving \ref{it:propsat1}$\Rightarrow$\ref{it:propsat2}. Assume that $X$ has the saturation property. Since $\Omega$ is $\sigma$-finite, it follows from \cite[Theorem~67.4]{Za67} that there exists an increasing sequence of measurable sets $F_n\subseteq\Omega$ with $\ind_{F_n}\in X$ and $\bigcup_{n=1}^\infty F_n=\Omega$. Note that while this result is stated for normed spaces, the proof also holds without change in quasi-normed spaces.

For \ref{it:propsat2}$\Rightarrow$\ref{it:propsat3}, we define
\[
u:=\sum_{n=1}^\infty \frac1{(2K_X)^n}\cdot\frac{\ind_{F_n}}{1+\nrm{\ind_{F_n}}_X}.
\]
By the Riesz--Fischer property, we have $u\in X$. So \ref{it:propsat3} follows from the fact that $u>0$ on $\Omega$.

For \ref{it:propsat3}$\Rightarrow$\ref{it:propsat4}, let $g \in L^0(\Omega)$ such that $\|fg\|_{L^1(\Omega)}=0$ for all $f\in X$. In particular, we have $\nrm{u g}_{L^1(\Omega)}=0$. This means that $u g=0$ a.e. and hence, since $u>0$ a.e., we must have $g=0$ a.e., as desired.

It remains to show \ref{it:propsat4}$\Rightarrow$\ref{it:propsat1}. Assume that $X$ does not have the saturation property. Then there is a set $E\subseteq\Omega$ of positive measure such that $\ind_F\notin X$ for all $F\subseteq E$ of positive measure. For $f\in X$, consider the sets $F_n:=\{x\in E:|f(x)|>\frac{1}{n}\}$ for $n\geq 1$. Since
\[
\ind_{F_n}\leq nf\in X,
\]
it follows from the ideal property of $X$ that $\ind_{F_n}\in X$ for all $n\geq 1$. But this means that $\mu(F_n)=0$ and hence,
\[
\mu\hab{\{x\in E:|f(x)|>0\}}=\mu\has{\bigcup_{n=1}^\infty F_n}\leq \sum_{n=1}^\infty\mu(F_n)=0.
\]
As this means that every function $f\in X$ vanishes a.e. on $E$, we have 
$\int_\Omega\! |f| \ind_E\,\mathrm{d}\mu = 0$ for all $f \in X$.
Since $\ind_E\neq0$, this proves the result by contraposition.
\end{proof}

\begin{remark}
By Proposition \ref{prop:weakorderunit} and our discussion in Subsection \ref{subs:idealprop}, using the terminology from Banach lattice theory, one could equivalently define a (quasi)-Banach function space as a (quasi)-Banach lattice of measurable functions such that:
\begin{itemize}
\item $X$ is an order ideal in $L^0(\Omega)$;
\item $X$ has  a weak order unit.
\end{itemize}
\end{remark}

\begin{remark}\label{rem:defBS}
    After the first chapter on general Banach function spaces, the book of Bennett and Sharpley \cite{BS88} is mainly focused on the case of so-called \emph{rearrangement-invariant} Banach function spaces, i.e. Banach function spaces $X$ such that any  $f,g\in X$ with the same distribution function have equal norm. In such spaces, it is easy to see that the saturation property is equivalent to the assumption \eqref{eq:illegal} and therefore also to  \eqref{eq:illegaldual} by Theorem \ref{thm:dualofbfsisbfs} below. This explains the choice for the ``simpler'' setup of Banach function spaces using \eqref{eq:illegal} and \eqref{eq:illegaldual} in \cite{BS88}.
\end{remark}

\begin{remark}\label{rem:bBFS}
Let $(\Omega,\mu)$ be a metric measure space.
    Ball quasi-Banach function spaces, as introduced in \cite{SHYY17}, satisfy the saturation property. Indeed, the sets $F_n=B(x,n)$ satisfy property \ref{it:propsat2} in Proposition~\ref{prop:weakorderunit}, where $B(x,n)$ denotes the ball around a point $x\in \Omega$ with radius $n$. In particular, every ball quasi-Banach function space is a quasi-Banach function space. Since the measure space over which a quasi-Banach function space is defined does not necessarily need to have a metric structure, the notion of a quasi-Banach function space is more general than that of a ball quasi-Banach function space.

    Conversely, if $X$ is a quasi-Banach function space on which the Hardy--Littlewood maximal operator is bounded, it is easy to see that  $\ind_B \in X$ for all balls $B$. So in this specific case, $X$ is automatically a ball quasi-Banach function space. 
\end{remark}

\section{Properties of Banach function spaces}
Having discussed the definition of quasi-Banach function spaces at length in the previous section, we will discuss some basic properties of quasi-Banach function spaces in this  section. We start by introducing the notion of K\"othe duality, which is the notion of duality within the category of Banach function spaces. 

Next, we will discuss important lattice properties that a Banach function space has through analogues of the classical convergence theorems in integration theory. We will start by discussing the Fatou property, which is the replacement of the monotone convergence theorem and Fatou's lemma. Then we discuss the notion of order-continuity, which serves as a replacement of the dominated convergence theorem.

Finally, we will discuss how the theory naturally includes weighted spaces (such as weighted Lebesgue spaces) in its definition through the saturation property. We show two ways of considering weights; by adding them to the underlying measure space, or by considering them as a multiplier.

\subsection{Duality}\label{subsect:duality} Duality arguments play an important role in mathematical analysis. For example, when working in $L^p(\R^d)$, duality often allows one to translate results for $1\leq p\leq 2$ to $2\leq p\leq \infty$ and vice versa. Unfortunately, the Banach dual $X^*$ of a quasi-Banach function space $X$ is not necessarily isomorphic to a space of functions. For example, the Banach dual of $L^\infty(\R^d)$ is a space of measures. 

Motivated by this phenomenon, we define the \emph{K\"othe dual} or \emph{associate space} $X'$ of a quasi-Banach function space $X \subseteq L^0(\Omega)$ as the the space
\[
X':=\{g\in L^0(\Omega):fg\in L^1(\Omega)\text{ for all $f\in X$}\}.
\]
For $g \in X'$ we define 
\[
\nrm{g}_{X'}:=\sup_{\nrm{f}_X=1}\|fg\|_{L^1(\Omega)},
\]
which is a norm on $X'$. Indeed, as shown in Proposition \ref{prop:weakorderunit}, the saturation property ensures (and is equivalent to the statement) that $\nrm{g}_{X'} = 0$ if and only if $g=0$ a.e. Moreover, $\nrm{g}_{X'}<\infty$. Indeed, suppose $\nrm{g}_{X'}=\infty$. Then, for each $n\geq 1$, there is an $f_n\in X$ with $\nrm{f_n}_X=1$ for which
$
\|f_n g\|_{L^1(\Omega)}>K_X^n n^3.
$
By the Riesz--Fischer property, we have $$F:=\sum_{n=1}^\infty\frac{|f_n|}{K_X^n n^2}\in X.$$ However, since also
\[
\|Fg\|_{L^1(\Omega)}\geq \frac{\nrm{f_n g}_{L^1(\Omega)}}{K_X^n n^2}>n
\]
for all $n\geq 1$, we deduce that $Fg\notin L^1(\Omega)$. By contraposition, we conclude that $\nrm{g}_{X'}<\infty$ for all $g \in X'$.

\bigskip

The ideal property of $L^1(\Omega)$ implies that $X'$ also satisfies the ideal property. Moreover, $X'$ satisfies the Fatou property (see Remark \ref{rem:def}\ref{it:def2} or Subsection \ref{subs:Fatou} for the definition), which implies the Riesz--Fischer property and, hence, by Proposition \ref{prop:complete} that $X'$ is complete. Indeed, if $0\leq g_n \uparrow g$ for $(g_n)_{n\geq 1}$ in $X'$ and $\sup_{n\geq 1}\nrm{g_n}_{X'}<\infty$, then, by the monotone convergence theorem, for every $f\in X$ we have $fg \in L^1(\Omega)$ with
\[
\|f g \|_{L^1(\Omega)}=\sup_{n \geq 1} \|f g_n  \|_{L^1(\Omega)} \leq \|f\|_X  \sup_{n \geq 1} \|g_n  \|_{X'},
\]
and therefore $g \in X'$ with $\nrm{g}_{X'}=\sup_{n\geq 1}\nrm{g_n}_{X'}$.

\bigskip

The map $f \mapsto \int_\Omega fg \dd\mu$ defines a bounded linear functional on $X$ for every $g \in X'$. Thus, we can naturally identify $X'$ with a closed subspace of $X^*$. Indeed, we have the following result:
\begin{proposition}\label{prop:kothedualembeddingdual}
Let $X$ be a quasi-Banach function space over $(\Omega,\mu)$. The embedding $\iota\colon X'\hookrightarrow X^\ast$ given by
\[
\iota(g)(f):=\int_\Omega\!fg\,\mathrm{d}\mu, \qquad f \in X,
\]
satisfies
$
\|\iota(g)\|_{X^\ast}=\|g\|_{X'}
$ for all $g \in X'$.
\end{proposition}
\begin{proof}
For $g \in X'$ we have
\[
\|\iota(g)\|_{X^\ast}=\sup_{\|f\|_X=1}\left|\int_\Omega\!fg\,\mathrm{d}\mu\right|\leq\sup_{\|f\|_X=1}\int_\Omega\!|fg|\,\mathrm{d}\mu=\|g\|_{X'},
\]
so it remains to show $\|g\|_{X'}\leq\|\iota(g)\|_{X^\ast}$.
Fix $f\in X$ and define $\widetilde{f}:=|fg|g^{-1}$ where $g$ is non-zero and zero elsewhere. Then $\widetilde{f}g=|fg|$ and $|\widetilde{f}|\leq|f|$ so that by the ideal property of $X$ we have $\widetilde{f}\in X$ with
\[
\int_\Omega\!|fg|\,\mathrm{d}x=|\iota(g)(\widetilde{f})|\leq\|\iota(g)\|_{X^\ast}\|\widetilde{f}\|_X\leq\|\iota(g)\|_{X^\ast}\|f\|_X.
\]
Taking a supremum over all $f\in X$ with $\|f\|_X=1$ proves the result.
\end{proof}

Next, we wish to determine when $X'$ is a Banach function space. We have already shown that $X'$ has the ideal property and is complete. Therefore,  to check that $X'$ is a Banach function space, it suffices to show that $X'$ has the saturation property. This, however, turns out to not always be the case. Indeed, for $X=L^p(\Omega)$ with $0<p<1$ we have $L^p(\Omega)'=\{0\}$, which is not saturated. Our goal is to characterize for which quasi-Banach function spaces $X$ the associate space $X'$ is a Banach function space.

When $X$ is a Banach function space, $X^*$ is non-trivial by the Hahn--Banach theorem. In fact, it turns out that in this case $X'$ is automatically saturated, and, hence, is a Banach function space:
\begin{theorem}\label{thm:dualofbfsisbfs}
Let $X$ be a Banach function space over $(\Omega,\mu)$. Then $X'$ is also a Banach function space over $(\Omega,\mu)$.
\end{theorem}
\begin{proof}
By the above discussion, we need only prove that $X'$ satisfies the saturation property. This follows from \cite[Theorem~71.4(a)]{Za67}
\end{proof}

This result allows us to prove the following characterization of when $X'$ is a Banach function space for a quasi-Banach function space $X$:
\begin{theorem}\label{thm:saturateddual}
Let $X$ be a quasi-Banach function space over $(\Omega,\mu)$. Then the following are equivalent:
\begin{enumerate}[(i)]
    \item\label{it:dualsaturated1} $X'$ is a Banach function space over $(\Omega,\mu)$;
    \item\label{it:dualsaturated2} There is a Banach function space $E$ over $(\Omega,\mu)$ such that $X\hookrightarrow E$.
\end{enumerate}
\end{theorem}

\begin{proof}
To prove \ref{it:dualsaturated1}$\Rightarrow$\ref{it:dualsaturated2}, assume $X'$ is a Banach function space, so $X''$ is well-defined. Since $X\hookrightarrow X''$ with 
\[
\|f\|_{X''}=\sup_{\|g\|_{X'}=1}\|fg\|_{L^1(\Omega)}\leq\sup_{\|g\|_{X'}=1}\|f\|_X\|g\|_{X'}=\|f\|_X
\]
and $X$ has the saturation property, $X''$ has the saturation property as well. Thus, $X''$ is a Banach function space. Therefore \ref{it:dualsaturated2} is satisfied with $E=X''$.

For the converse, let $E$ be as in \ref{it:dualsaturated2} and let $C>0$ such that $\|f\|_E\leq C\|f\|_X$ for all $f\in X$. Then, by Theorem~\ref{thm:dualofbfsisbfs}, the space $E'$ has the saturation property. Since
\[
\|g\|_{X'}=\sup_{\|f\|_X\leq 1}\|fg\|_{L^1(\Omega)}\leq \sup_{\|f\|_E\leq C}\|fg\|_{L^1(\Omega)}=C\|g\|_{E'}
\]
for all $g\in E'$, we have that $E'\subseteq X'$. This means that $X'$ has the saturation property as well, proving \ref{it:dualsaturated1}.
\end{proof}

\subsection{The Fatou property}\label{subs:Fatou}
The Fatou property is essentially an $X$-version of the monotone convergence theorem and Fatou's lemma from integration theory, and reduces back to these classical results in the case that $X = L^1(\Omega)$. 

\begin{definition}
Let $X$ be a quasi-Banach function space over $(\Omega,\mu)$. We say that $X$ satisfies the \emph{Fatou property} if it satisfies the following condition: if $0\leq f_n \uparrow f$ for $(f_n)_{n\geq 1}$ in $X$ and $\sup_{n\geq 1}\nrm{f_n}_X<\infty$, then $f \in X$ and $\nrm{f}_X=\sup_{n\geq 1}\nrm{f_n}_X$.
\end{definition}

The Fatou property is equivalent to an $X$-version of Fatou's lemma, which explains the nomenclature.
\begin{lemma}\label{lem:Fatou}
Let $X$ be a quasi-Banach function space over $(\Omega,\mu)$. Then $X$ has the Fatou property if and only if for any sequence of positive-valued functions  $(f_n)_{n\geq 1}$ in $X$ with $\liminf_{n \to \infty}\nrm{f_n}_X <\infty$, one has $\liminf_{n \to \infty} f_n \in X$ and
$$
\nrmb{\liminf_{n \to \infty} f_n}_X \leq \liminf_{n \to \infty}\,\nrm{f_n}_X.
$$
\end{lemma}

\begin{proof} We only need to prove the forward implication, for which we
 define $g_n  = \inf_{k\geq n} f_k$ for $n \geq 1$. Then, by the ideal property, we have $g_n \in X$ and for all $m\geq n$ we have
$
\nrm{g_n}_X \leq \nrm{f_m}_X.
$
In particular,
$$
\nrm{g_n}_X \leq \inf_{m \geq n} \,\nrm{f_m}_X.
$$
Since $0\leq g_n \uparrow \liminf_{m\to \infty} f_m$, it follows from the Fatou property that $\liminf_{m \to \infty} f_m \in X$ and
$$
\nrmb{\liminf_{m \to \infty} f_m}_X = \nrmb{\lim_{n \to \infty} g_n}_X =\sup_{n\geq 1} \,\nrm{g_n}_X \leq \sup_{n\geq 1}\inf_{m \geq n} \nrm{f_m}_X = \liminf_{n \to \infty}\,\nrm{f_n}_X.
$$
This finishes the proof.
\end{proof}

When $X$ is a quasi-Banach function space, it follows from the monotone convergence theorem that $X'$ satisfies the Fatou property. This proves one direction of the so-called Lorentz--Luxemburg theorem:
\begin{theorem}\label{thm:lorentzluxemburg}
    Let $X$ be a Banach function space over $(\Omega,\mu)$. Then $X$ satisfies the Fatou property if and only if we have $X''=X$ with equal norm.
\end{theorem}
For a full proof of this result we refer the reader to \cite[Theorem~71.1]{Za67}. In particular this result implies that for a Banach function space $X$, its K\"othe dual $X'$ is norming, i.e., we have
\[
\|f\|_{X}=\sup_{\|g\|_{X'}=1}\|fg\|_{L^1(\Omega)}.
\]
We point out that this result means that the Fatou property is equivalent to reflexivity in terms of K\"othe duality.

Even if a quasi-Banach function space does not have the Fatou property, it always embeds into one that does. This explains why one often assumes the Fatou property as part of the definition of a quasi-Banach function space.

\begin{proposition}\label{prop:fatouembedding}
Let $X$ be a quasi-Banach function space over $(\Omega,\mu)$. Then there is a quasi-Banach function space $Y$ over $(\Omega,\mu)$ that satisfies the Fatou property for which $X\hookrightarrow Y$ with $\|f\|_Y\leq\|f\|_X$ for all $f\in X$.
\end{proposition}
According to Zaanen (see \cite[Section~66]{Za67}), the following construction was originally introduced by G.~G.~Lorentz in an unpublished work.
\begin{proof}
We let $Y\subseteq L^0(\Omega)$ denote the space of $f\in L^0(\Omega)$ for which there exists a sequence $(f_n)_{n\geq 1}$ in $X$ for which $0\leq f_n\uparrow |f|$ a.e. and $\sup_{n\geq 1}\|f_n\|_X<\infty$. We equip this space with the Lorentz quasi-norm
\[
\|f\|_Y:= \inf\cbraceb{ \sup_{n\geq 1}\,\|f_n\|_X : 0\leq f_n \uparrow \abs{f}}.
\]
For $f\in X$ we can set $f_n:=|f|$ for $n\geq 1$, so we have $X\subseteq Y$ with $\|f\|_Y\leq\|f\|_X$.
 
In the case that $\nrm{\,\cdot\,}_X$ is a norm, the proof that $\nrm{\,\cdot\,}_Y$ is also a norm satisfying the Fatou property can be found in \cite[Section~66]{Za67}. These proofs remain valid, mutatis mutandis, when $\nrm{\,\cdot\,}_X$ is a quasi-norm.
\end{proof}
\begin{remark}
When $X'$ is saturated, one can note that $X''$ is a Banach function space with the Fatou property that $X$ embeds into. As a matter of fact, it is shown in \cite[Theorem~71.2]{Za67} that when $X$ is a Banach function space, then the $Y$ constructed in the above proof is equal to $X''$. Remarkably, the above construction remains valid even when $X'$ is not saturated (in which case $\|\cdot\|_{X''}$ would only be a seminorm).
\end{remark}
\begin{remark}
    Sometimes one only has a weaker version of the Fatou property:
    \begin{itemize}
\item \textit{Weak Fatou property:} There is a $C>0$ such that if $0\leq f_n \uparrow f$ for $(f_n)_{n\geq 1}$ in $X$ and $\sup_{n\geq 1}\nrm{f_n}_X<\infty$, then $f \in X$ and $\nrm{f}_X\leq C\sup_{n\geq 1}\nrm{f_n}_X$.
\end{itemize}
    For a quasi-Banach function space with the weak Fatou property,    one actually has $Y=X$ isomorphically in Proposition \ref{prop:fatouembedding}. A typical example where one runs into the weak Fatou property, is when one passes to an equivalent quasi-norm on a space with the Fatou property. For example, this happens when one equips a quasi-Banach function space with the Aoki-Rolewicz $p$-norm as defined in \eqref{eq:Aoki}. If one then wants to retain the Fatou property, one can then apply the construction in Proposition \ref{prop:fatouembedding} on this $p$-norm and check that this is again a $p$-norm, this time with the Fatou property.
\end{remark}

Finally, we show that quasi-Banach function spaces with the Fatou property are, in some sense, maximal. In particular, we show that if a quasi-Banach function space $X$ isometrically embeds as a proper subspace into a quasi-Banach function space $Y$, then $X$ cannot have the Fatou property. For example, this means that $c_0$ does not have the Fatou property, as it is a proper closed subspace of the Banach function space $\ell^\infty$.

\begin{proposition}\label{prop:noproperclosedsubspace}
    Suppose  $X$ and $Y$ are quasi-Banach function spaces over $(\Omega,\mu)$, $X \subseteq Y$ and $\nrm{f}_X=\nrm{f}_Y$ for all $f \in X$. If $X$ has the Fatou property, then $X=Y$.
\end{proposition}

\begin{proof}
Let $f \in Y$, let $0<u \in X$ be a weak order unit, and define $g_n := \min(\abs{f},nu )$ for $n \geq 1$. Then, by the ideal property of $X$, we have $g_n \in X$ for all $n \geq 1$. Moreover, $g_n \uparrow \abs{f}$ a.e. and, by the ideal property of $Y$,
\[
\sup_{n\geq 1}\|g_n\|_X=\sup_{n\geq 1}\|g_n\|_Y\leq \|f\|_Y.
\]
Hence, $\abs{f} \in X$ by the Fatou property of $X$ so that $f \in X$ by the ideal property of $X$.
\end{proof}

\subsection{Order continuity} Having dealt with $X$-valued versions of the monotone convergence theorem and Fatou's lemma through the Fatou property, we now wish to discuss the third main convergence theorem of integration theory: the dominated convergence theorem. To make sense of this theorem in a quasi-Banach function space setting, we introduce the notion of order convergence. We say that a sequence $(f_n)_{n\geq 1}$ in $X$ \emph{order converges} to $f \in X$ if there is a sequence $(g_n)_{n\geq 1}$ in $X$ such that $g_n \downarrow 0$ and $\abs{f-f_n} \leq g_n$ for all $n\geq 1$. Using this terminology, the dominated convergence theorem is equivalent to the statement that order convergence implies norm convergence for $X=L^1(\Omega)$, which can be seen by taking $g_n = \sup_{k\geq n} \abs{f-f_k}$ for $n \geq 1$.

Not all quasi-Banach function spaces have the property that order convergence implies norm convergence. Indeed, if $X = L^\infty(\Omega)$, order convergence corresponds to pointwise a.e. convergence for a bounded sequence of functions, whereas norm convergence corresponds to uniform a.e. convergence. This motivates the following definition.

\begin{definition}
 A quasi-Banach function space $X$ over $(\Omega,\mu)$ is called \emph{order-continuous} if for sequences $(f_n)_{n\geq 1}$ in $X$ with $f_n\downarrow 0$ pointwise a.e. we have $\nrm{f_n}_X\downarrow 0$.
\end{definition}
We note that in an order-continuous quasi-Banach function space, order convergence implies norm convergence, which explains the nomenclature. Rephrasing, a quasi-Banach function space $X$ is order-continuous if and only if an $X$-version of the dominated convergence theorem holds, i.e. for any sequence $(f_n)_{n\geq 1}$ in $X$ such that $f_n \to f$ pointwise a.e. and $\abs{f_n} \leq g \in X$ for all $n \geq 1$, it follows that 
$$\lim_{n \to \infty}\nrm{f_n - f}_X = 0.$$

As already noted, $L^\infty(\Omega)$ is not order-continuous. In particular, the sequence space $\ell^\infty$ is not order-continuous. This space is actually the prototypical space that is not order-continuous in the following sense: Any quasi-Banach function space that is not order continuous contains a (lattice) isomorphic copy of $\ell^\infty$. For Banach lattices, this can be found in \cite{MN73} (see also \cite[Theorem 1.a.7]{LT79}), which can be adapted to the quasi-Banach function space setting using the Aoki--Rolewicz theorem.

\bigskip

Various authors use  a different, but equivalent notion instead of order-continuity. A quasi-Banach function space $X$ is said to have \emph{absolutely continuous quasi-norm} if, for all $f \in X$ and for all decreasing sequences of measurable sets $(E_n)_{n\geq 1}$ with $\ind_{E_n}\downarrow 0$ a.e., we have $\nrm{f\ind_{E_n}}_X\downarrow 0$.

\begin{proposition}\label{prop:abscontordercont}
Let $X$ be a quasi-Banach function space over $(\Omega,\mu)$. Then $X$ is order-continuous if and only if $X$ has an absolutely continuous quasi-norm.
\end{proposition}

\begin{proof}
It is clear that order-continuity implies that $X$ has an absolutely continuous quasi-norm by taking $f_n = |f|\ind_{E_n}$.  For the converse, let $(f_n)_{n\geq 1}$ be a sequence in $X$ with $f_n \downarrow 0$ pointwise a.e. Take $\varepsilon>0$, let $u\in X$ be a weak order unit, and
define
\[
E_n:=\cbraceb{x\in \Omega:f_n(x)u(x)^{-1}>(2K_X\nrm{u}_X)^{-1}\varepsilon }.
\]
Since $f_n u^{-1}\downarrow 0$ pointwise a.e., we know that $E_n$ decreases to a set of measure zero. By the absolute continuity of the quasi-norm of $X$, we can find an $N\geq 1$ such that
\begin{equation*}
\|f_1\ind_{E_N}\|_X<\frac{\varepsilon}{2K_X}.
\end{equation*}
By the ideal property, this implies for all $n \geq N$
\begin{align*}
\|f_n\|_X&\leq K_X\|f_n\ind_{\Omega\backslash E_N}\|_X+K_X\|f_n\ind_{E_N}\|_X\\
&\leq K_X \cdot  (2K_X\nrm{u}_X)^{-1}\varepsilon \cdot \|u \ind_{\Omega\backslash E_{N}}\|_X+K_X\|f_1\ind_{E_N}\|_X\\
&<\frac{\varepsilon}{2}+\frac{\varepsilon}{2}=\varepsilon.
\end{align*}
The assertion follows.
\end{proof}

We say that the measure space $(\Omega,\mu)$ is \emph{separable} if there is a countable collection of measurable sets $\mc{A}$ such that for every measurable set $E\subseteq \Omega$ with $\mu(E)<\infty$ and every $\varepsilon>0$ one can find an $A \in \mc{A}$ with $\mu(A\Delta E)<\varepsilon$. Note that, in particular, the Lebesgue measure on  $\R^d$ is separable.
For separable $(\Omega,\mu)$, the order-continuity of a quasi-Banach function space $X$ over $(\Omega,\mu)$ implies the separability of $X$. 

\begin{proposition}\label{prop:separable}
    Let $X$ be a quasi-Banach function space over a separable measure space $(\Omega,\mu)$. If $X$ is order-continuous, then $X$ is separable.
\end{proposition}

\begin{proof}
Let $\mc{A}$ be a countable collection of measurable sets 
such that for every measurable set $E\subseteq \Omega$ with $\mu(E)<\infty$ and every $\varepsilon>0$ there is an $A \in \mc{A}$ with $\mu(A\Delta E)<\varepsilon$. By Proposition \ref{prop:weakorderunit}, there is a weak order unit $u \in X$, i.e. a function $u\in X$ such that  $u>0$ a.e. We claim that the countable set of functions
\begin{equation*}
   \cbraces{\sum_{k=1}^n a_k\cdot u \ind_{A_k}: a_k \in \Q\oplus i\Q, \, A_k \in \mc{A}} \subseteq X
\end{equation*}
is dense in $X$. Indeed, for any $f \in X$, we can find a sequence of simple functions $(f_n)_{\geq 1}$ such that $\abs{f_n} \leq \abs{f}u^{-1}$ for all $n\geq 1$ and $f_n \to fu^{-1}$ pointwise a.e. Therefore, by the order-continuity of $X$, we have that $f_nu \to f$ in $X$. Hence, by the density of $\Q$ in $\R$ and the quasi-triangle inequality, it suffices to show that for all measurable $E\subseteq \Omega$  there is a sequence $(A_n)_{n\geq 1}$ in $\mc{A}$ such that $\lim_{n\to \infty}\nrm{u\ind_E-u\ind_{A_n}}_X=0$ Moreover, since $(\Omega,\mu)$ is $\sigma$-finite, it suffices to consider $\mu(E)<\infty$

Fix a measurable $E \subseteq \Omega$ with $\mu(E)<\infty$ and let let $(A_{n})_{n\geq 1}$ be a sequence of measurable sets in $\mc{A}$ such that $\mu(A_{n} \Delta E) \to 0$ as $n\to \infty$, i.e. $\ind_{A_n} \to \ind_E$ (locally) in measure. Then there is a subsequence such that $u\ind_{A_{n_k}}\to u\ind_E$ pointwise a.e. By the order-continuity of $X$, we conclude that $\lim_{k\to \infty} \nrm{u\ind_{A_{n_k}}-u\ind_E}_X = 0$, finishing the proof.
\end{proof}

\begin{remark}
    The converse of Proposition \ref{prop:separable} also holds: if $X$ is a separable quasi-Banach function space over $(\Omega,\mu)$, then $X$ is order-continuous and $(\Omega,\mu)$ is separable. The order-continuity of $X$ follows from the fact that $X$ contains an isomorphic copy of $\ell^\infty$ if it is not order-continuous and for the separability of $(\Omega,\mu)$, one can adapt the proof of \cite[Theorem 1.5.5]{BS88}
    \end{remark}

By Proposition~\ref{prop:kothedualembeddingdual}, $X'$ can be identified with a closed subspace in $X^\ast$. In the next proposition we will characterize when $X'= X^*$.

\begin{proposition}\label{prop:dualordercont}
    Let $X$ be a quasi-Banach function space over $(\Omega,\mu)$. 
    \begin{enumerate}[(i)]
        \item\label{it:dualordercont1} If $X$ is order-continuous, then $X' = X^*$.
        \item\label{it:dualordercont2}  If $X$ is a Banach function space and $X' = X^*$, then $X$ is order-continuous.
    \end{enumerate}
\end{proposition}

\begin{proof}
    For \ref{it:dualordercont1} assume that $X$ is order-continuous and let $u \in X$ be a weak order unit, i.e. $u>0$ a.e. Take $x^* \in X^*$ and for all measurable $E \subseteq \Omega$ define $\lambda(E) = x^*(\ind_E u)$. Then $\lambda$ is a complex measure, since for disjoint, measurable $E_1,E_2,\ldots \subseteq \Omega$ we have
    \begin{align*}
        \lambda\has{\bigcup_{n=1}^\infty E_n} = \lim_{N\to \infty} \sum_{n=1}^N \lambda\ha{E_n} + x^*\has{\sum_{n = N+1}^\infty \ind_{E_n}u} = \sum_{n=1}^\infty \lambda\ha{E_n}.
    \end{align*}
    where the last step follows from $\sum_{n = N+1}^\infty \ind_{E_n}u \downarrow 0$ pointwise a.e. as $N \to \infty$ and the order continuity of $X$.
    Moreover, note that $\lambda$ is absolutely continuous with respect to $u\,\mathrm{d}\mu$, so by the Radon--Nikodym theorem there is a $g \in L^0(\Omega)$ such that $x^*(\ind_E u) = \lambda(E) = \int_Eg u\dd\mu$ for all measurable $E \subseteq \Omega$.

    Now let $f \in X$ be arbitrary and let $(f_n)_{n\geq 1}$ be a sequence of simple functions such that $\abs{f_n} \leq \abs{f}u^{-1}$ for all $n\geq 1$ and $f_n \to fu^{-1}$ pointwise a.e. By the order-continuity of $X$, we have $f_n u \to f$ in $X$ and thus, by the dominated convergence theorem,
    $$
    x^*(f) = \lim_{n \to \infty} x^*(f_n u) = \lim_{n \to \infty} \int_\Omega f_n g u\dd\mu = \int_\Omega fg \dd\mu.
    $$
    We conclude that $g \in X'$ and $\iota(g) = x^*$, which shows that $X^* =X'$.

    \medskip

    For \ref{it:dualordercont2} assume that $X$ is a Banach function space and $X' = X^*$. Let $(f_n)_{n\geq 1}$ be a sequence in $X$ such that $f_n \downarrow 0$ pointwise a.e. For any $g \in X'$ we have, by the dominated convergence theorem, that
    $$
    \lim_{n \to \infty}  \int_\Omega f_n g\dd\mu = 0.
    $$
    Since  $X' = X^*$, we deduce that $\cbrace{f_n:n\geq 1}\cup {0}$ is weakly closed. Thus, by the Hahn--Banach separation theorem, its convex hull is norm closed. Therefore, for any $\varepsilon>0$, there are $a_1,\ldots,a_n \geq 0$ with $\sum_{k=1}^n a_k=1$ such that
    $
    \nrm{\sum_{k=1}^n a_k f_k}_X<\varepsilon.
    $
    Since $(f_n)_{n\geq 1}$ is decreasing, this implies that $\nrm{f_j}_X<\varepsilon$ for all $j\geq n$, finishing the proof.
\end{proof}

We note that, if $X$ is a quasi-Banach function space, it can happen that $X^* = \cbrace{0}$. In this case the assumption $X'=X^*$ is trivial, which explains the need for the assumption that $X$ is a Banach function space in Proposition \ref{prop:dualordercont}\ref{it:dualordercont2}. For an example of a quasi-Banach function space $X$ with trivial dual and which is not order-continuous, we refer the reader to \cite[Example 2.19]{ORS08}. 

\bigskip

We end this subsection with a corollary on the connection between order-continuity and reflexivity.

\begin{corollary}\label{cor:reflexiveordercontinuous}
        Let $X$ be a Banach function space over $(\Omega,\mu)$. Then $X$ is reflexive if and only if $X$ has the Fatou property and $X$ and $X'$ are order-continuous.
\end{corollary}
\begin{proof}
If $X$ has the Fatou property and $X$ and $X'$ are order-continuous, we immediately obtain $$X^{**}=X'^{*} = X'' = X$$ by Proposition \ref{prop:dualordercont}\ref{it:dualordercont1} and Theorem \ref{thm:lorentzluxemburg}, so $X$ is reflexive. 

For the converse, assume that $X$ is reflexive. By Theorem~\ref{thm:dualofbfsisbfs} the space $X'$ is saturated, so by Proposition~\ref{prop:kothedualembeddingdual} and Proposition \ref{prop:weakorderunit}\ref{it:propsat4}, we have
\[
(X')^{\top}=\big\{f\in X:\int_{\R^d}\!fg\,\mathrm{d}x=0\text{ for all $g\in X'$}\big\}=\{0\}.
\]
As $X$ is reflexive and $X'$ is closed, this proves that $X'=X^\ast$. Thus, Proposition \ref{prop:dualordercont}\ref{it:dualordercont2} implies that $X$ is order-continuous. Moreover, we obtain
\[
X'^\ast=X^{\ast\ast}=X.
\]
By Proposition~\ref{prop:kothedualembeddingdual} we have $X''\subseteq X'^\ast$ and, as in the proof of Theorem~\ref{thm:saturateddual}, we have $X\subseteq X''$, both with embedding constant $1$ . We conclude that actually
\[
X''=X'^\ast=X.
\]
The first equality implies that $X'$ is order-continuous by Proposition \ref{prop:dualordercont}\ref{it:dualordercont2}, and the second equality implies that $X$ has the Fatou property by Theorem~\ref{thm:lorentzluxemburg}. This proves the result.
\end{proof}

\subsection{Weighted Banach function spaces}\label{subsec:weights}
In this final subsection, we want to make clear that the saturation property naturally allows one to consider weighted spaces without having to change any of the defining properties of a quasi-Banach function space to weighted versions (cf. \cite{CMM22}).
Moreover, we will provide a general strategy which can be used to transfer results in the literature for quasi-Banach function spaces (and their K\"othe duals) assumed to contain all indicator functions of sets of finite measure to results for quasi-Banach function spaces satisfying the saturation property.

To do so, we discuss two ways of introducing a weight to a quasi-Banach function space: as a multiplier and as a change of measure. For the multiplier viewpoint, we  take a weight $0<w\in L^0(\Omega)$, a quasi-Banach function space $X$ over $(\Omega,\mu)$, and define a new space $X(w)$ as the space of those $f\in L^0(\Omega)$ for which $fw\in X$, equipped with the quasi-norm
\[
\|f\|_{X(w)}:=\|fw\|_X.
\]
This is again a quasi-Banach function space over $(\Omega,\mu)$:
\begin{proposition}\label{prop:weightmultiplier}
Let $X$ be a quasi-Banach function space over $(\Omega,\mu)$ and let $0<w\in L^0(\Omega)$. Then $X(w)$ is a quasi-Banach function space over $(\Omega,\mu)$ with $K_{X(w)}=K_X$ and
\[
X(w)'=X'(w^{-1}).
\]
Moreover, if $X$ has the Fatou property or is order-continuous, then the same holds for $X(w)$.
\end{proposition}
\begin{proof}

We observe that the map $f\mapsto fw^{-1}$ is an order preserving isometric isomorphism between $X$ and $X(w)$. Hence, the ideal, Riesz--Fischer and Fatou properties, as well as order-continuity,  respectively, are possessed by $X(w)$ if and only if they are by $X$. Similarly, for the saturation property, note that if $0<u\in X$ is a weak order unit, then its image under this map $uw^{-1}\in X(w)$ is also a weak order unit. This concludes the proof of the first result.

For the equality $X(w)'=X'(w^{-1})$, we note that
\begin{align*}
\|g\|_{X(w)'}&=\sup_{\|f\|_{X(w)}=1}\int_\Omega\!|f|w|g|w^{-1}\,\mathrm{d}\mu
=\sup_{\|h\|_{X}=1}\int_\Omega\!|h||g|w^{-1}\,\mathrm{d}\mu\\
&=\|gw^{-1}\|_{X'}=\|g\|_{X'(w^{-1})}.
\end{align*}
This proves the result.
\end{proof}

Instead of adding a weight as a multiplier, we can also take a weight $0<w\in L^0(\Omega)$ and consider it as a change of measure through
\[
w(E):=\int_E\!w\,\mathrm{d}\mu,
\]
which we will also denote as $w\,\mathrm{d}\mu$. The measure space $(\Omega,w\,\mathrm{d}\mu)$ is again a $\sigma$-finite measure space: since $L^1(\Omega)(w)$ is saturated by Proposition~\ref{prop:weightmultiplier}, it follows from Proposition~\ref{prop:weakorderunit} that there exists a sequence of measurable sets $F_n\subseteq\Omega$ increasing to $\Omega$ with $\ind_{F_n}\in L^1(\Omega)(w)$, i.e., $w(F_n)<\infty$, for all $n\geq 1$.

In a similar vein, a quasi-Banach function space $X$ over $(\Omega,\mu)$ is also a quasi-Banach function space over $(\Omega,w\,\mathrm{d}\mu)$. Note that K\"othe duality depends on the underlying measure space, so when there is such a change of measure, we will write $X^\dag$ to denote the K\"othe dual with respect to this new measure.
\begin{proposition}\label{prop:weightmeasure}
Let $X$ be a quasi-Banach function space over $(\Omega,\mu)$ and let $0<w\in L^0(\Omega)$. Then $X$ is also a quasi-Banach function space over $(\Omega,w\,\mathrm{d}\mu)$. Moreover, we have
\[
X^\dag=X'(w).
\]
If $X$ has the Fatou property or is order-continuous with respect to $(\Omega,\mu)$, then the same holds with respect to $(\Omega,w\,\mathrm{d}\mu)$.
\end{proposition}
\begin{proof}
Since the ideal property, Riesz--Fischer property, Fatou property, order-continuity, and the existence of a weak order unit all only depend on the null sets of the underlying measure space, the fact that $(\Omega,\mu)$ and $(\Omega,w\,\mathrm{d}\mu)$ have the same null sets proves that $X$ has any of these respective properties with respect to $(\Omega,w\,\mathrm{d}\mu)$ if and only if it has them with respect to $(\Omega,\mu)$. As for the duality result, we have
\[
\|g\|_{X^\dag}=\sup_{\|f\|_X=1}\int_\Omega\!|fg|w\,\mathrm{d}\mu
=\|gw\|_{X'}=\|g\|_{X'(w)},
\]
as desired.
\end{proof}

\begin{example}
    When it comes to the Lebesgue spaces, the idea of changing measure takes the following form: for $p\in[1,\infty)$, the space $L^p(\Omega,w\,\mathrm{d}\mu)$ can be seen as either a space over $(\Omega,w\,\mathrm{d}\mu)$, in which case the K\"othe dual is given by $L^{p'}(\Omega,w\,\mathrm{d}\mu)$, or as a space over $(\Omega,\mu)$, in which case the K\"othe dual is given by $L^{p'}(\Omega,w^{1-p'}\,\mathrm{d}\mu)$. Both of these approaches appear in the literature, and it seems to be a matter of taste and context which one is preferred.

For the multiplier approach, for $p\in(0,\infty)$ we have
\[
\|fw\|_{L^p(\Omega)}=\Big(\int_\Omega\!|f|^pw^p\,\mathrm{d}\mu\Big)^{\frac{1}{p}},
\]
which shows that $L^p(\Omega)(w)=L^p(\Omega,w^p\,\mathrm{d}\mu)$. When $p\in[1,\infty]$, Proposition~\ref{prop:weightmultiplier} yields
\[
(L^p(\Omega)(w))'=L^{p'}(\Omega)(w^{-1}),
\]
which is equal to $L^{p'}(\Omega,w^{-p'}\,\mathrm{d}\mu)$ when $p>1$. 
\end{example}

\begin{remark}
    In Lebesgue spaces with $p<\infty$, both the multiplier and the change of measure approach yield the same theory up to a change of weight $w\mapsto w^p$. When $p=\infty$, the multiplier approach is preferable. As a matter of fact, while the change of measure approach is classically used more frequently, we would argue that the multiplier approach does not only result in a theory that includes the case $p=\infty$ in a satisfying manner, but also leads to a more symmetric theory altogether. Therefore, in our view, it is more intuitive and easier to work with.
\end{remark}

The situation quickly becomes more complicated when we venture beyond the classical (weighted) Lebesgue spaces. For example, when dealing with weak type spaces $L^{p,\infty}(\Omega)$ for $p\in(0,\infty)$, both approaches yield different kinds of weighted spaces. Indeed, while both the spaces 
\[
L^{p,\infty}(\Omega)(w),\quad L^{p,\infty}(\Omega,w^p\,\mathrm{d}\mu)
\]
contain the space $L^p(\Omega)(w)=L^p(\Omega,w^p\,\mathrm{d}\mu)$ by Chebyshev's inequality, they are generally not equal.
For example, set $\Omega=\R^d$ equipped with the Lebesgue measure and define $w(x):=|x|^{-\frac{d}{p}}$. Then $\ind_{\R^d}\in L^{p,\infty}(\R^d)(w)$, but $\ind_{\R^d}\notin L^{p,\infty}(\Omega,w^p\,\mathrm{d}x)$.

As we saw in Theorem~\ref{prop:weightmeasure}, a change of measure for $X$  will lead to a weight as a multiplier in the associate space. Hence, to fully understand the properties of a space with respect to weights, usually both of these approaches need to be understood.

\bigskip

Since many results in the literature are proven for quasi-Banach function spaces $X$ over $(\Omega,\mu)$ with the property that for all measurable $E\subseteq\Omega$ with $\mu(E)<\infty$ we have $\ind_E\in X$ and $\int_E\!|f|\,\mathrm{d}\mu<\infty$ for all $f\in X$ (i.e., $\ind_E\in X'$), one might wonder if there is a general method of replacing these arguments with arguments that only require the saturation property. This is indeed the case, as we will outline next.

By Proposition~\ref{prop:weakorderunit}, the saturation property means that there is a function $u\in X$ that satisfies $u>0$ a.e. If we now take \emph{any} measurable set $E\subseteq\Omega$, then the ideal property of $X$ implies that also $\ind_E u\in X$. Thus, the weighted space $X(u)$ contains the indicator function of all measurable subsets of $\Omega$; not just the ones with finite measure. Moreover the map $f \mapsto u^{-1} f$ is a positive isometry from $X$ to $X(u)$.

If the space $X(u)$ is considered a Banach function space over $(\Omega,w\,\mathrm{d}\mu)$ for a particular weight $w$, then for all measurable set $E\subseteq\Omega$ we also have that $\ind_E\in X(u)^\dag$, i.e.,
\[
\int_E\!|f|w\,\mathrm{d}\mu<\infty
\]
for all $f\in X(u)$. The following result is a quasi-Banach function space version of this (cf. \cite[Lemma~7.4]{CNS03}).
\begin{proposition}\label{prop:weightkothe}
Let $X$ be a quasi-Banach function space over $(\Omega,\mu)$ and let $0<u\in X$ be a weak order unit. Then $\ind_E\in X(u)$ for \emph{all} measurable $E\subseteq\Omega$.

If $X'$ is a Banach function space, then there is a weight $0<w\in L^1(\Omega)$ for which the space $X(u)$, considered as a quasi-Banach function space over $(\Omega,w\,\mathrm{d}\mu)$, additionally satisfies $\ind_E\in X(u)^\dag$ for \emph{all} measurable $E\subseteq\Omega$, i.e.,
\[
\int_E\!|f|w\,\mathrm{d}\mu<\infty
\]
for all $f\in X(u)$.
\end{proposition}
\begin{proof}
For the first assertion, note that $\ind_E\in X(u)$ if and only if $\ind_Eu\in X$. Since $\ind_Eu\leq u$, this follows from the ideal property.

For the second assertion, let $0<v\in X'$ be a weak order unit and set $w:=uv\in L^1(\Omega)$. Then it follows from Proposition~\ref{prop:weightmultiplier} and Proposition~\ref{prop:weightmeasure}  that 
\[
X(u)^\dag=X(u)'(w)=X'(u^{-1}w)=X'(v),
\]
so using the same argument as before with $X$ replaced by $X'$ and $u$ replaced by $v$ proves the result.
\end{proof}
\begin{remark}
In connection to Proposition~\ref{prop:weightkothe} we want to note that if $X$ is a quasi-Banach function space for which $X'$ is saturated, then picking weak order units $0<u\in X$, $0<v\in X'$ and setting $w:=uv\in L^1(\Omega)$, we have the inclusions
\begin{equation}\label{eq:weightinclusions1}
L^\infty(\Omega)(w^{-1})\subseteq X(v^{-1}) ,\quad X(u)\subseteq L^1(\Omega)(w).
\end{equation}
If $X$ is a Banach function space and $(\Omega,\mu)$ is a finite measure space, then we can actually find a $0<u\in X$ for which
\begin{equation}\label{eq:weightinclusions2}
L^\infty(\Omega)\subseteq X(u)\subseteq L^1(\Omega).
\end{equation}
Indeed, the Lozanovskii factorization theorem states that $L^1(\Omega)=X\cdot X'$ (see \cite{Lo69, Gi81}). This implies that we can find $0<u\in X$, $0<v\in X'$ such that $1=uv$. Then \eqref{eq:weightinclusions1} implies \eqref{eq:weightinclusions2}.
\end{remark}

Proposition~\ref{prop:weightkothe} can essentially be used as a ``patch'' 
to extend results appearing in the literature that assume that the space contains the indicator functions of sets of finite measure, to the more general class of spaces with the saturation property. We do wish to point out that, in our opinion, this solution is inelegant. One might as well do the proof correctly in the first place by arguing through a weak order unit directly. As a matter of fact, it is our hope that this survey will function as a ``patch'' for the future literature.

\subsection*{Acknowledgements}
We are grateful to Ben de Pagter for reading our manuscript to make sure we did, in fact, not do Banach function spaces wrong.
We would like to thank Dalimil Pe\v{s}a for a discussion that let to Remark \ref{rem:defBS}, Jordy van Velthoven for bringing the PhD thesis \cite{Vo15} to our attention and Dachun Yang for a discussion that led to a clarification of Remark \ref{rem:bBFS}. Moreover, we thank the anonymous referees for several insightful comments pertaining to missed references or historical inaccuracies.
Finally, we would like to thank Mark Veraar for encouraging us to write this manuscript, as well suggesting the title.

\bibliography{bieb}
\bibliographystyle{alpha}
\end{document}